\documentclass[12pt]{article}

\usepackage[cp1250]{inputenc}
\usepackage{lmodern}
\usepackage{a4wide}
\usepackage{amsfonts}
\usepackage{bbm}
\usepackage[leqno]{amsmath}
\usepackage{amssymb}
\usepackage{amsthm}
\usepackage{mathrsfs}
\usepackage{ifpdf}
\usepackage[dvips]{graphicx}
\usepackage{epsfig}
\usepackage{latexsym}
\usepackage{authblk}
\usepackage[mathscr]{euscript}
\usepackage[T1]{fontenc}
\usepackage{graphics}
\usepackage{enumerate}
\usepackage{longtable}
\usepackage{fancyhdr}
\usepackage{booktabs} 
\usepackage{multicol} 
\usepackage{multirow}
\usepackage[bookmarksopen]{hyperref}
\usepackage{natbib}
\newcommand{\Z}{\mathbb{Z}}

\newcommand{\p}{\mathbb{P}}

\newcommand{\dd}{\mathrm{d}}
\newcommand{\im}{\mathrm{i}}

\newcommand{\sds}{\mathrm{SDS}}

\theoremstyle{plain}
\newtheorem{theorem}{Theorem}[section]

\newtheorem{lemma}[theorem]{Lemma}

\theoremstyle{definition}

\begin{document}

\title{On discrete approximations of stable distributions}
\renewcommand\Authands{ and }

\author[a,b]{Lenka Sl\'amov\'a \thanks{Corresponding author; Email: \texttt{slamova.lenka@gmail.com};}}
\author[a]{Lev B. Klebanov \thanks{Email: \texttt{levbkl@gmail.com}}}

\affil[a]{\small{\textit{Department of Probability and Mathematical Statistics, Charles University in Prague, Czech Republic}}}
\affil[b]{\small{\textit{Department of Applied Mathematics and Statistics, Stony Brook University, NY, USA}}}

\maketitle

\abstract{In some fields of applications of stable distributions, especially in economics, it appears, that data have distributions similar to stable in a large region, but do not have such heavy tails. Our aim in this note is to propose several methods of approximation of stable distributions by some discrete distributions, which may have different tail behavior. In a sense the introduced distributions form an alternative to tempered stable distributions that combine Gaussian and stable behavior.}

\section{Introduction}
Stable distributions are on the rise in financial applications since \cite{mandelbrot} noted that Gaussian distribution does not provide a good fit for financial returns that exhibit leptokurtic behavior and heavy tails. However the infinite variance of stable distributions and the fact that financial returns have heavier tails on a short time scale and almost Gaussian on a long scale brings into question the appropriateness of the stable model of returns. \cite{grabchak} studied this paradox and suggested that a more appropriate model for financial returns has tempered heavy tails. They show, using the pre-limit theorem by \cite{klebanov_prelimit}, that the sum of a large number of independent and identically distributed random variables behave as a stable random variable even though the tails of the random variables are not heavy. 

Stable distributions with exponentially tempered tails have been considered in the literature under different names - truncated L\'evy flights (\cite{koponen}), CGMY model (\cite{carr}) and finally tempered stable distributions (\cite{rosinski}). Tempered stable distributions appear by exponential tilting of the L\'evy measure of stable distributions. The resulting distributions have finite moments of any order and exponential tails.

In this note, we focus on discrete approximations of stable distributions both with heavy tails and truncated heavy tails and therefore offer an alternative to the stable and tempered stable distributions. In the second and third section, we introduce two approximations of the stable characteristic function leading to discrete distributions that were introduced in \cite{slamova2012}. Also we study two different approximations with Gaussian tails, appearing as a result of truncation and tempering the heavy tails of the discrete stable distributions. In the fourth section, we study a discrete approximation resulting from discretizing the L\'evy measure of stable distributions. We obtain a discrete distribution that allow for the index of stability to be an arbitrary positive number. 

\section{First approximation}
Let us first consider the case of symmetric $\alpha$-stable distributions. Their characteristic functions are given by the following formula $$f(t) = \exp\left\{-\sigma^{\alpha} |t|^{\alpha}\right\},$$ with $\alpha \in (0,2]$ being the index of stability and $\sigma > 0$ being the scaling parameter. For arbitrary $t$ we write
$|t|^{\alpha} = (t^2)^{\gamma},$ where $\gamma = \alpha/2$. Let us use the following approximation. We have $$t^2 = \lim_{a \to 0} \frac{2}{a^2} (1-\cos(at)),$$ therefore let us write $t^2 \sim \frac{2}{a^2} (1-\cos(at)),$ as $a \to 0$. Hence the characteristic function of symmetric $\alpha$-stable distribution can be approximated as
\begin{align*}
\log f(t) & = -\sigma^{\alpha} |t|^{\alpha} \sim \log g(t,a) = - \sigma^{2\gamma} \frac{2^\gamma}{a^{2\gamma}} (1-\cos(at))^{\gamma},
\end{align*} 
for small values of $a$.

\begin{lemma}
The function $$g(t,a) = \exp\left\{- \sigma^{2 \gamma} \frac{2^\gamma}{a^{2\gamma}} (1-\cos(at))^{\gamma}\right\}$$ is a characteristic function of a distribution given on the lattice $a\Z = \{0,\pm a, \pm 2a,\dots\}$ for any positive $a$.
\end{lemma}
\begin{proof}
We can rewrite $g(t,a)$ as $g(t,a) = \exp\{-\lambda(1-h(t,a))\}$, where $$h(t,a) = 1- (1-\cos (at))^{\gamma} = \sum_{k=1}^{\infty} \binom{\gamma}{k} (-1)^{k-1} \cos(at)^k.$$ The series coefficients are positive for $\gamma \in (0,1]$, moreover $h(0,a) = 1$, $h(t,a)$ is periodic with period $2\pi a$, hence the function $h(t,a)$ is a characteristic function of a random variable on $a\Z$. Therefore $g(t,a)$ is a characteristic function of compound Poisson random variable with intensity of jumps $\lambda$ and jumps in $a\Z$ with characteristic function $h(t,a)$.
\end{proof}

It is clear that $$\lim_{a \to 0} g(t,a) = f(t),$$ and therefore $g(t,a)$ can be considered as discrete approximation of $f(t)$ for a sufficiently small $a$. This distribution with $a=1$ was introduced in \cite{slamova2012} by considering a discrete analogue of the stability property $X = n^{-1/\alpha} (X_1 + X_2 + \dots + X_n)$ and called symmetric discrete stable (SDS) distribution.

It is obvious from the construction of $\gamma$-symmetric discrete stable distribution that it belongs to the domain of normal attraction of $2\gamma$-stable distribution. From the known characterization of the domain of attraction of stable distributions (see, for example, \cite{ibragimov}), a $\sds$ random variable must satisfy the following tail assumptions as $x \to \infty$
\begin{equation}
\lim_{x \to \infty} x^{2 \gamma} \p(|X| > x) = \left\{
\begin{array}{ll} 
\lambda\frac{a^{2\gamma}}{2^{\gamma}}\frac{1}{\Gamma(1-2\gamma)\cos(\pi \gamma)} & \text{if  } \gamma \ne \frac{1}{2} \\ 
\lambda\frac{a^{2\gamma}}{2}\frac{2}{\pi} & \text{if  } \gamma = \frac{1}{2} 
\end{array}\right.
\end{equation}

So far we have introduced a discrete approximation of the symmetric $\alpha$-stable distribution that has the same tail behavior. Another approximation leading to a distribution with exponential tails can be obtained in the following way. Let us consider a function
$$ g(t,a, M) = \exp \left\{-\lambda \sum_{k=1}^{M}(-1)^k \binom{\gamma}{k} \cos(at)^k + \lambda \sum_{k=1}^{M}(-1)^k \binom{\gamma}{k} \right\}.$$ For sufficiently large values of $M$ this function can be considered an approximation of the function $g(t,a)$, as $$\lim_{M \to \infty} g(t,a,M) = g(t,a).$$ So it is also a discrete approximation of symmetric $\alpha$-stable distribution, as $$\lim_{a \to 0} \lim_{M \to \infty} g(t,a,M) = f(t).$$ However we cannot exchange the order of the limits. 

Both characteristic functions $g(t,a)$ and $g(t,a,M)$ are infinitely divisible as they correspond to compound Poisson distributions. The distributions of jumps are given by $h(t,a) = 1-(1-\cos(at))^{\gamma}$ and $h(t,a,M) = 1-\sum_{k=1}^M(-1)^k\binom{\gamma}{k}\cos(at)^k$ + $\sum_{k=1}^M(-1)^k\binom{\gamma}{k}$ respectively. It can be verified that these distributions have no mass at 0 and the second distribution has truncated jumps in absolute value larger than $aM$. We will therefore call the distribution given by characteristic function $g(t,a,M)$ truncated symmetric discrete stable distribution.
 
The characteristic function is an entire function hence the tails of the truncated symmetric discrete stable distribution behave like $o(\exp(-bx))$, as $x\to \infty$, for all $b>0$ by  the Raikov's theorem (\cite{linnik}). The truncated symmetric discrete stable distribution thus belongs to the domain of normal attraction of Gaussian distribution and as such has finite variance. It follows from the pre-limit theorem of \cite{klebanov_prelimit} that for not too large values of $n$ the sum $S_n  =n^{-1/2\gamma}(X_1 + \dots + X_n)$ behaves like symmetric $\alpha$-stable distribution with $\alpha = 2\gamma$. This property is due to the truncation of the bigger jumps, therefore the distribution behaves like stable distribution in the middle, and like Gaussian on the tails. 

\section{Second approximation}
In the previous section, we introduced a discrete approximation of symmetric $\alpha$-stable distribution. Here we give a discrete approximation of $\alpha$-stable distribution with index of stability $\alpha \in (0,1)$ and skewness $\beta \in [-1,1]$. The characteristic function of strictly $\alpha$-stable distribution with skewness parameter $\beta$ and scale parameter $\sigma>0$ is given by $$f(t) = \exp\left\{-\sigma^{\alpha}|t|^{\alpha}\left(1-\im \beta \mathrm{sign}(t) \tan \frac{\pi \alpha}{2}\right)\right\}.$$ We can rewrite this as 
\begin{align*}
\log f(t) = - \lambda_1 (-it)^{\alpha} -\lambda_2(it)^{\alpha},
\intertext{where}
\lambda_1 = \frac{\sigma^{\alpha}}{\cos \frac{\pi \alpha}{2}} \frac{1+\beta}{2}, \quad
\lambda_2 = \frac{\sigma^{\alpha}}{\cos \frac{\pi \alpha}{2}} \frac{1-\beta}{2},
\end{align*}
We use the following approximation: $it \sim (1-e^{-\im a t})/a$ as $a \to 0$ and  $-it = \sim (1-e^{\im a t})/a$ as $a \to 0$, therefore the characteristic function of $\alpha$-stable distribution can be approximated by a characteristic function of a discrete distribution as
$$\log f(t) \sim \log g(t,a) = -  \frac{\lambda_1}{a^{\alpha}} (1-e^{\im a t})^{\alpha}-  \frac{\lambda_2}{a^{\alpha}} (1-e^{-\im a t})^{\alpha}, \quad \text{as} \quad a \to 0.$$  
This distribution for $a = 1$ was introduced in \cite{slamova2012} as discrete stable distribution and it was shown there that $g(t,a)$ is a characteristic function only for $\alpha \in (0,1]$. From the construction of the approximation we see that $$\lim_{a\to 0} g(t,a) = f(t).$$ Discrete stable distribution has therefore the same behavior of tails as $\alpha$-stable distribution and it is again infinitely divisible.

We can obtain yet another discrete approximation of $\alpha$-stable distribution with Gaussian tails by tempering the tails of discrete stable distribution. Because discrete stable distribution is a compound Poisson distribution with intensity $\lambda_1 + \lambda_2$ and distribution of jumps with characteristic function $$h(t,a) = 1- \frac{\lambda_1}{\lambda_1+\lambda_2}(1-e^{\im a t})^{\alpha} - \frac{\lambda_2}{\lambda_1+\lambda_2}(1-e^{-\im a t})^{\alpha},$$ the L\'evy-Khintchine representation of discrete stable characteristic function takes the following form 
$$\log g(t,a) = \int_{-\infty}^{\infty} \left(e^{\im a t x} - 1\right) \nu(\dd x),$$ where $\nu(\dd x)$ is the L\'evy measure,
\begin{align*}
\nu(\dd x) = (\lambda_1 + \lambda_2) \sum_{k=-\infty}^{\infty} p_k \delta_{ak}(\dd x),\\
\intertext{where}
p_k = \left\{\begin{array}{ll} 
\frac{\lambda_1}{\lambda_1 + \lambda_2} (-1)^{k+1} \binom{\alpha}{k} & k>0\\
\frac{\lambda_2}{\lambda_1 + \lambda_2} (-1)^{k+1} \binom{\alpha}{|k|} & k<0\\
0 & k = 0,
\end{array}
\right.
\end{align*}
and $\delta_k$ is the Dirac measure, i.e.~$\delta_x(A) = 1$ if $x \in A$ and $0$ otherwise. The classical idea leading to tempered infinitely divisible distribution consists of exponential tempering of the corresponding L\'evy measure (\cite{rosinski}). We will use tempering function of the form $q(x) = e^{-\theta_1 x} \mathbf{1}_{x>0} + e^{-\theta_2 |x|} \mathbf{1}_{x<0}$. The tempered infinitely divisible distribution is then obtained by multiplying the L\'evy measure by this tempering function. As a result we obtain a distribution with characteristic function
\begin{multline*}
 \log g(t,a,\theta_1,\theta_2) = -\lambda_1 \left(1-e^{\im a t} e^{-\theta_1}\right)^{\alpha} - \lambda_2 \left(1-e^{-\im a t} e^{-\theta_2}\right)^{\alpha} \\ + \lambda_1 \left(1- e^{-\theta_1}\right)^{\alpha} + \lambda_2 \left(1- e^{-\theta_2}\right)^{\alpha}.
\end{multline*}

This characteristic function is an analytic function in the strip $\mathfrak{I}(t) \in (-\theta_2, \theta_1)$ and by the Raikov's theorem (\cite{linnik}) the tails are  $O(\exp(-b x))$, as $x \to \infty$, for all $b>0$. Therefore the tempered discrete stable distribution belong to the domain of normal attraction of Gaussian distribution. By the pre-limit theorem of \cite{klebanov_prelimit} we can show that for not too large values of $n$, the normalized sum $S_n = n^{-1/\alpha}(X_1 + \dots + X_n)$ behaves like $\alpha$-stable distribution.

\section{Third approximation}
Another way to find a discrete approximation of strictly stable distributions is to discretize its L\'evy (or spectral) measure. Stable distribution is an infinitely divisible distribution and as such has a L\'evy-Khintchine representation of its characteristic function. This representation takes the following form (see, for example, \cite[$\S$ 34]{zolotarev} or \cite{samorodnitsky})
\begin{equation}\label{LK_stable}
\log f(t) =  P \int_0^{\infty}\left(e^{\im t x} - 1 -\im t x \mathbf{1}_{|x|\leq 1}\right) \frac{\dd x}{x^{1+\alpha}} + Q \int_{-\infty}^{0}\left(e^{\im t x} - 1 -\im t x \mathbf{1}_{|x| \leq 1}\right) \frac{\dd x}{|x|^{1+\alpha}}, 
\end{equation}
where $P, Q >0$, $0 < \alpha < 2$. The L\'evy measure is $$U(\dd x) = \frac{P}{x^{1+\alpha}} \mathbf{1}_{x>0}(x) \dd x  + \frac{Q }{|x|^{1+\alpha}}\mathbf{1}_{x<0}(x) \dd x.$$ Every infinitely divisible random variable is a limit of compound Poisson random variables, and the L\'evy  measure $U(\dd x)$ express the intensity of jumps of size $x$. The term $\im t x \mathbf{1}_{|x|\leq 1}$ is the compensation of the small jumps that ensures that the integral converges. If we discretize the L\'evy measure we obtain discrete infinitely divisible distribution. This can be achieved by discretizing the integrals in \eqref{LK_stable}, as follows
\begin{multline}\label{LK_discrete}
\log f(t) \sim \log g(t,a) =  P \sum_{k=1}^{\infty}\left(e^{\im t a k} - 1\right) \frac{a}{(ak)^{1+\alpha}} \\
+ Q \sum_{k=-\infty}^{-1}\left(e^{\im t ak} - 1 \right) \frac{a}{|ak|^{1+\alpha}}, \quad \text{  as  } \quad a \to  0.
\end{multline}
Here we omit the compensation of small jumps. The distribution given by characteristic function $g(t,a)$ is still infinitely divisible, with the L\'evy measure given by 
$$V(\dd x) = \sum_{k=-\infty}^{\infty} P \frac{a}{(ak)^{1+\alpha}}\mathbf{1}_{k>0}(k) \delta_{ak}(\dd x) + Q \frac{a}{|ak|^{1+\alpha}}\mathbf{1}_{k<0}(k)\delta_{ak}(\dd x), \quad k \in \Z.$$ As a result, we obtain a distribution with characteristic function 
\begin{equation*}
\log g(t,a) = \frac{1}{a^{\alpha}} \left(P \mathrm{Li}_{1+\alpha}\left(e^{\im a t}\right) + Q \mathrm{Li}_{1+\alpha}\left(e^{-\im a t}\right) - (P+Q) \zeta(1+\alpha)\right),
\end{equation*}
where $\mathrm{Li}_{1+\alpha}(x)$ is the polylogarithm function and $\zeta(1+\alpha)$ is the Zeta function. It is interesting to note that $g$ is a characteristic function for all positive values of $\alpha$. We can rewrite this with $\sigma = (P+Q)$ and $\beta = \frac{P}{P+Q}$ to obtain
\begin{equation*}
g(t,a) = \exp \left\{\frac{\sigma}{a^{\alpha}} \left(\beta \mathrm{Li}_{1+\alpha}\left(e^{\im a t}\right) + (1-\beta) \mathrm{Li}_{1+\alpha}\left(e^{- \im a t}\right) - \zeta(1+\alpha)\right)\right\}.
\end{equation*}

By truncating the series in \eqref{LK_discrete} we obtain yet another approximation of $\alpha$-stable distribution by an entire characteristic function. 

\section*{Acknowledgements}
The paper was partially supported by Czech Science Foundation under the grants P402/12/12097 and P203/12/0665. The support by Mobility fund of Charles University in Prague and Karel Urb\'anek endowment fund is gratefully acknowledged.

\bibliography{mybib}{}

\end{document}